\documentclass[12pt]{amsart} 

\usepackage{amsmath,amsthm,amscd,amssymb,bm,mathtools, mathdots,mathrsfs,enumerate,cases}
\usepackage[arrow,matrix, all, cmtip]{xy}
\usepackage[plainpages,urlcolor=blue]{hyperref}

\theoremstyle{plain}
\newtheorem{thm}[subsection]{Theorem}
\newtheorem{lem}[subsection]{Lemma}
\newtheorem{prop}[subsection]{Proposition}
\newtheorem{cor}[subsection]{Corollary}
\newtheorem{conj}[subsection]{Conjecture}

\theoremstyle{definition}
\newtheorem{rk}[subsection]{Remark}

\newtheorem{prob}[subsection]{Problem}

\numberwithin{equation}{section}
\newcommand{\al}{\alpha}

\newcommand{\bb}{\mathbb}
\newcommand{\be}{\beta}

\newcommand{\Disc}{\mathrm{Disc}}

\newcommand{\ms}{\mathscr}

\newcommand{\p}{\partial}

\newcommand{\Res}{\mathrm{Res}}

\newcommand{\llangle}{\langle\hspace*{-1.5pt}\langle}
\newcommand{\rrangle}{\rangle\hspace*{-1.5pt}\rangle}

\begin{document}
\date{}

\title[Self-intersection number of negative curves]{Self-intersection Number of Negative Curves on Fermat Surfaces}
\author[ZHENJIAN WANG]{ZHENJIAN WANG}
\address{Hefei National Laboratory, Hefei 230088, China}
\email{wzhj@ustc.edu.cn}



\begin{abstract} 
We give an explicit formula for the self-intersection number of negative curves on Fermat surfaces. The formula offers us hints to prove the Bounded Negativity Conjecture for the Fermat surfaces by providing explicit negativity bounds, and it also offers the possibility of counterexamples to the conjecture.
\end{abstract}


\maketitle


\section{Introduction}

Let $\bb{P}^n$ be the complex projective space of dimension $n$. Let $X_d: x^d+y^d+z^d+w^d=0$ be the Fermat surface of degree $d$ in $\bb{P}^3$ with variables $x,y,z,w$, and let $D:x^d+y^d+z^d=0$ be the Fermat curve of degree $d$ in $\bb{P}^2$ with variables $x,y,z$. We shall study self-intersection number of negative curves on $X_d$. The main objective in this paper is to establish the following theorem.
\begin{thm}[Theorem \ref{main}]\label{BNCFermat}
Let $d\geq5$. Suppose that $C$ is a reduced and irreducible curve on the Fermat surface $X_d$ with self-intersection number $C^2<2-d<0$. Then the following properties hold.
\begin{enumerate}[(1)]
\item There exists a unique reduced and irreducible curve $E\subset\bb{P}^2$ defined by a degree $e$ polynomial $g(x,y,z)$, such that the cone in $\bb{P}^3$ determined by $g$ contains $C$. 
\item There exists $k\geq2$ and $k|d$ such that
\begin{align}\label{eq:main}
C^2=&\frac{d}{k}e^2-\frac{d(d-1)}{k}e-\frac{d}{k}\sum_{p\in D\cap E}\mu_p(E)\notag\\
   &+\sum_{[x,y,z]\in D\cap E}\mu_{[x,y,z,0]}(C)+\left(\frac{d}{k}-1\right)\#(D\cap E).
\end{align}
In particular, when $d$ is prime, we have $k=d$ and
\[
C^2=e^2-(d-1)e-\sum_{p\in D\cap E}\mu_p(E)+\sum_{[x,y,z]\in D\cap E}\mu_{[x,y,z,0]}(C).
\]
In the above formulae, $\mu_p(E)$ denotes the Milnor number of the plane curve $E$ at the point $p$, and $\mu_{[x,y,z,0]}(C)$ denotes the Milnor number of $C$ at $[x,y,z,0]$, and $\#(D\cap E)$ the number of points in the set $D\cap E$.
\end{enumerate}
\end{thm}

The technical condition $C^2<2-d$ in Theorem 1.1 is given to exclude exceptional cases where $C^2$ can be calculated explicitly. This ensures a concise formulation while maintaining generality. The exceptional cases are investigated in Section \ref{sec:str}.

The computation of self-intersection number of negative curves, as in Theorem 1.1, is closely connected with an old folklore conjecture among experts in algebraic surface theory, commonly known as the \textit{Bounded Negativity Conjecture}. We formulate this conjecture as follows.

\begin{conj}[Bounded Negativity Conjecture]\label{BNC1}
For any given smooth complex projective surface $X$, there exists a number $B_X\geq0$ which depends only on $X$ such that for any reduced curve $C$ on $X$, the self-intersection number $C^2$ is bounded below by $-B_X$, i.e., $C^2\geq-B_X$.
\end{conj}

The origins of this conjecture are unclear, but it can at least be traced back to the Italian mathematician Federigo Enriques (1871-1946), and it has been an open problem for one hundred years. For a more comprehensive overview of the history, we refer to \cite[p.1878]{BauHarKnuKueMueRouSze1307}.

An alternative formulation of the Bounded Negativity Conjecture is as follows.

\begin{conj}[Bounded Negativity Conjecture for integral curves]\label{BNC2}
For any given smooth complex projective surface $X$, there exists a number $b_X\geq0$ which depends only on $X$ such that for any reduced and irreducible curve $C$ on $X$, the self-intersection number $C^2$ is bounded below by $-b_X$, i.e., $C^2\geq-b_X$.
\end{conj}
Indeed, as proven in \cite[Prop.5.1, p.1891]{BauHarKnuKueMueRouSze1307}, we can take the number $B_X$ in Conjecture \ref{BNC1} as $B_X=\left(\rho(X)-1\right)b_X$, where $\rho(X)$ is the Picard number of $X$. We will refer to $b_X$ as the \textit{negativity bound} for $X$.

The result in Theorem \ref{BNCFermat} implies that under some conditions, the Bounded Negativity Conjecture is true for the Fermat surface $X_d$.
\begin{cor}[Corollary \ref{sufcond}]\label{cor:sufcond}
The Bounded Negativity Conjecture holds for the Fermat surface $X_d$ if one of the following conditions is satisfied:
\begin{enumerate}[(a)]
\item
\[
\liminf_{e\to\infty}\inf_{E}\left(e^2-(d-1)e-\sum_{p\in D\cap E}\mu_p(E)\right)\geq -\lambda
\]
for some constant $\lambda>0$ which may depend on $d$; or
\item
\[
\limsup_{e\to\infty}\sup_E\frac{\sum\limits_{p\in D\cap E}\mu_p(E)}{e^2}<1;
\]
\end{enumerate}
where, in the above two formulas, the curve $E$ ranges over all irreducible projective plane curves of degree $e$.
\end{cor}

On the other hand, we hope that the formula for $C^2$ in Theorem \ref{BNCFermat} also provides the possibility of counterexamples to the Bounded Negativity Conjecture on Fermat surfaces, although no explicit counterexamples have been constructed.

In the sequel, we will only investigate the self-intersection number of reduced and irreducible curves. For convenience, we denote $C$ such a curve on a given smooth projective surface $X$. 

It is evident that the Bounded Negativity Conjecture holds for $X$ with $-K_X\geq0$, as can be seen by applying the adjunction formula:
\[
K_X\cdot C+C^2=2p_a(C)-2\Longrightarrow C^2\geq-2.
\]
In particular, for $d\leq 4$, the bounded negativity of $X_d$ follows from this argument, since the canonical divisor of $X_d$ is $K_{X_d}=(d-4)H$ where $H$ is the hyperplane section. However, for $d\geq5$, the canonical divisor $K_{X_d}$ is ample and $X_d$ is of general type. To the author's knowledge, there are currently no effective methods available to deal with the Bounded Negativity of such surfaces, except for the weak Bounded Negativity Conjecture proved in \cite{Hao1908}.

Another natural approach to establish the Bounded Negativity for a surface is to show that there are only finitely many negative curves on it. However, it is proved in \cite[Thm.4.1]{BauHarKnuKueMueRouSze1307} that for any $m>0$, there exists a smooth projective surface which admits infinitely many  smooth irreducible curves of self-intersection number $-m$. For the Fermat surface $X_d$, we do not know whether the number of negative curves on $X_d$ is finite or not.

Other strategies to prove the Bounded Negativity Conjecture include investigating the non-isomorphic surjective endomorphisms or establishing dominance by other surfaces that have bounded negativity \cite[Prop.2.1]{BauHarKnuKueMueRouSze1307}, exploring the birational invariance of bounded negativity and the $H$-constant \cite{BauDirHarHuiLunPokSze1501}, using the equivalence between bounded negativity and the boundedness of the denominators in the Zariski decomposition \cite{BauPokSch1712}, discussing the weighted Bounded Negativity Conjecture \cite{LafPok2001,GalMonMor2408,GalMonPer2501}, and studying the connections between the Bounded Negativity Conjecture and other unsolved problems, such as SHGH conjecture or the bounded cohomology conjecture \cite{BauBocCooDirDumHarJabKnuKueMirRoeSchZse1205,CilKnuLesLozMirMusTes1708}. 

However, despite providing invaluable insights in understanding the Bounded Negativity Conjecture, these methods either deal with only very special cases, e.g., rational surfaces or K3 surfaces, or just offer general frameworks that reduce the conjecture to other open questions. They rely on general formulae, such as the adjunction formula and Riemann-Roch formula, to establish bounds for the intersection number $C^2$ of a reduced and irreducible curve $C$, without using specific information about $C$, such as its equation. In addition, the negativity bound, even if exists, remains implicit in most cases.

 Can we estimate $C^2$ from the equation of the curve $C$?  Can we give a concrete example of a surface with the Bounded Negativity by performing explicit computation through the defining equation of the surface, thereby yielding an explicit negativity bound?

To explore these questions, we consider smooth surfaces in $\bb{P}^3$, and give the following conjecture.

\begin{conj}\label{cd}
For any $d\geq1$, there exists a constant $c_d$ that depends only on $d$ such that all smooth surfaces of degree $d$ in $\bb{P}^3$ have negativity bound $c_d$.
\end{conj}

Since all the degree $d$ surfaces in $\bb{P}^3$ are contained in a connected family, they can be deformed into one another. We further conjecture that the bounded negativity property is invariant under deformations.

\begin{conj}\label{definv}
Suppose $X$ and $Y$ are two smooth complex projective surfaces belonging to a single family of complex projective surfaces, then the Bounded Negativity Conjecture is true for $X$ if and only if it is true for $Y$.
\end{conj}

Assuming the conjectures \ref{cd},\ref{definv} above, we are motivated to prove the validity of the Bounded Negativity Conjecture for a specific surface in $\bb{P}^3$. Theorem \ref{BNCFermat} is proved for this purpose. 

Note that a general surface in $\bb{P}^3$ has Picard number 1, which follows from the Noether-Lefschetz theorem; hence for a general surface $X\subset\bb{P}^3$, we may take $b_X=0$. Therefore, the Bounded Negativity Conjecture for surfaces in $\bb{P}^3$ follows from Conjecture \ref{definv} without explicit computation. However, the Fermat surface does not have Picard number 1, and our computation in Theorem \ref{BNCFermat} is of independent interest, particularly in providing a nontrivial explicit bound $c_d$ in Conjecture \ref{cd}.

The formulae in Theorem \ref{BNCFermat} show how the Bounded Negativity Conjecture is connected with the singularity theory of curves. The key term in formula \eqref{eq:main} is the sum $\sum\limits_{p\in D\cap E}\mu_p(E)$. It is notable that this sum is taken over the intersection points of $D$ and $E$, not all the singular points of $E$.

To conclude the introduction, we give an outline of the proof of Theorem \ref{BNCFermat}. We represent the Fermat surface $X_d$ as a branched Galois cover of $\bb{P}^2$:
\[
\rho:\ X_d\to\bb{P}^2,\quad[x,y,z,w]\mapsto[x,y,z].
\]
Given a negative curve $C$ on $X_d$, let $E=\rho(C)$ and $\ms{C}=\rho^*E$. Then $E$ is a reduced and irreducible curve in $\bb{P}^2$, whose defining equation is denoted by $g=g(x,y,z)$. It is clear that $g$ can be seen as a section of the line bundle $\ms{O}_{X_d}\left((\deg g)H\right)$ which vanishes on $C$. Then
we prove that $\ms{C}$ is a reduced curve on $X_d$; see Proposition \ref{intnum}. Moreover, $\ms{C}$ is of the following form
\[
\ms{C}=C_1+\cdots+C_k,
\]
with $C_i^2=C$ for all $i=1,\ldots, k$. This result holds because $C_i$'s lie in the orbit of $C$ under the Galois group action for $\rho$. It follows that
\[
kC^2=d(\deg g)^2-\sum_{i\neq j}C_i\cdot C_j.
\]
As a result, to compute $C^2$, we need to calculate the sum of intersection numbers on the right hand side. Finally, the intersection numbers can be calculated in a local way at any intersection point. Locally, the curve $\ms{C}$ can be represented as the curve in $(\bb{C}^3,0)$ defined by
\[
h(u,v)=0,\ l(u,v)+\omega^d=0.
\]
We will first solve $v$ as Puiseux series of $u$ from the first equation, and then use $l\left(u,v(u)\right)+\omega^d=0$ to investigate the intersection multiplicities between different irreducible components of $\ms{C}$. This is the most technical part in this paper and will be completed in Proposition \ref{intnum} and Section \ref{sec:intnum}.

\bigskip

The author would like to thank the organizers of International Symposium on Singularities and Applications, Dec. 9--13, 2024 in Sanya, China. The stimulating atmosphere of the conference inspired him to conceive the idea of performing computation using Puiseux series. The author also thanks Professor Feng Hao for helpful discussions.

\section{Curves on surfaces of Sebastiani-Thom type}
If $X\subset\bb{P}^3$ is a smooth surface defined by the following equation:
\[
X: f(x,y,z)+w^d=0,
\]
where $f$ is a form of degree $d$, then $X$ is said to be of \textit{Sebastiani-Thom type}. The Fermat surface $X_d$ is a typical example. 

We show that a curve on $X$ is ``almost'' defined by a homogeneous polynomial in the following result.

\begin{prop}\label{curvST}
Let $X: f(x,y,z)+w^d=0$ be a smooth surface in $\bb{P}^3$ and $C$ a reduced and irreducible curve on $X$. Then there exists a unique irreducible form $g(x,y,z)$ up to a multiplicative constant, such that the divisor on $X$ cut out by $g$ can be written as
\[
(g)=m(C_1+\cdots+C_k),
\]
where $C_1=C,C_2,\ldots, C_k$ are distinct reduced and irreducible curves and $m\geq1$ is an integer,
and, in addition, the self-intersection numbers $C_i^2, i=1,\ldots, k$ are all equal to $C^2$.
\end{prop}
\begin{proof}
Consider the following projection morphism
\[
\rho: X\to\bb{P}^2,\ [x,y,z,w]\mapsto[x,y,z].
\]
It is a branched Galois covering map with Galois group $\bb{Z}_d=\bb{Z}/d\bb{Z}$, the Galois action given by
\[
a\cdot[x,y,z,w]=[x,y,z,e^{2\pi\sqrt{-1}a/d}w],\ \forall a\in\bb{Z}_d.
\]

Denote $E=\rho(C)$. Then $E$ is a reduced and irreducible curve in $\bb{P}^2$. 

Let $g(x,y,z)$ be an irreducible defining equation for $E$. Then the inverse image $\rho^{-1}(E)$ is exactly the curve defined by $g=0$ on $X$ (here we see $g$ as a homogeneous polynomial in $x,y,z,w$ restricted to $X$). Write the divisor $(g)$ cut out by $g$ on $X$ as
\[
(g)=m_1C_1+\cdots+m_kC_k,\ C_1=C,
\]
where $C_1,\ldots, C_k$ are reduced and irreducible curves on $X$ and $m_1,\ldots, m_k\geq1$. 

Since the action $\bb{Z}_d$ for the morphism $\rho$ is a Galois action, $\rho^{-1}(E)$ is exactly the union of the curves in the $\bb{Z}_d$-orbit of $C$. It follows that for each pair $i\neq j$, there exists an element $\al$ in the Galois group $\bb{Z}_d$ such that $\al\cdot C_i=C_j$.
Now it is obvious that $g$ is a $\bb{Z}_d$-invariant polynomial in $(x,y,z,w)$; thus, 
\[
m_i=\mathrm{ord}_{C_i}(g)=\mathrm{ord}_{\al\cdot C_i}(\al\cdot g)=\mathrm{ord}_{C_j}(g)=m_j.
\]
Hence $m_1=\cdots=m_k=m$ for some $m\geq1$. In addition, since $C_i,i>1$ are all isomorphic to $C$ by automorphisms of $X$, we have $C_i^2=C^2$ for all $i=1,\ldots, k$.

It remains to show that the polynomial $g$ is unique up to a constant factor. Indeed, $g$ must be the defining polynomial of the irreducible curve $E$ since $g$ depends only the variables $x,y,z$, without involving the variable $w$. 
\end{proof}

\section{Strategy of Proof of Theorem \ref{BNCFermat}}\label{sec:str}
Let $X_d: x^d+y^d+z^d+w^d=0$ be the Fermat surface in $\bb{P}^3$ and $C\subset X_d$ a reduced and irreducible curve. We shall investigate the lower bound of the self-intersection number $C^2$. 

By Proposition \ref{curvST}, there exists an irreducible polynomial $g(x,y,z)$ which cuts out a divisor on $X_d$ of the following form
\[
(g)=m(C_1+\cdots+C_k),\ C_1=C
\]
and $C_i^2=C^2$ for $i=1,\ldots,k$. 

If $k=1$, we have $m^2C^2=d(\deg g)^2$ and thus $C^2\geq0$. Therefore, to investigate the Bounded Negativity Conjecture for $X_d$, we assume $k\geq2$ in the sequel.

Recall that $D:x^d+y^d+z^d=0$ is the Fermat curve. Let $ E:g(x,y,z)=0$
be the plane curve defined by $g$ in $\bb{P}^2$, and let $e=\deg g=\deg E$.

Furthermore, denote $\ms{C}$ the scheme on $X_d$ defined by $g=0$:
\[
\ms{C}: g(x,y,z)=0.
\]
Then $C_1,\ldots, C_k$ are the irreducible components of $\ms{C}$.

\subsection{Strategy}
From the equation
\[
(g)=m(C_1+\cdots+C_k),
\]
we get
\begin{align}\label{eq:C2}
de^2&=m^2(C_1+\cdots+C_k)^2\notag\\
    &=m^2\sum_{i=1}^kC_i^2+m^2\sum_{i\neq j}C_i\cdot C_j\notag\\
    &=m^2kC^2+m^2\sum_{i\neq j}C_i\cdot C_j,
\end{align}
where $C_i\cdot C_j$ denotes the intersection number of the curves $C_i$ and $C_j$ on $X_d$.

Our aim is to derive a lower bound for $C^2$. From formula \eqref{eq:C2}, we are led to establish an upper bound for the sum
\[
\sum_{i\neq j}C_i\cdot C_j,
\]
which will yield a lower bound for the following difference
\[
de^2-m^2\sum_{i\neq j}C_i\cdot C_j.
\]
If this difference can be bounded from below as a function in $e\geq1$ when $d$ is viewed as a constant, then the Bounded Negativity Conjecture for $X_d$ will follow.

In the main part of this paper, we will calculate the intersection numbers $C_i\cdot C_j$ through local computation methods. The key results are summarized in Proposition \ref{intnum}, and the lengthy proof will be presented in Section \ref{sec:intnum}.

For the local computation, fix an arbitrary intersection point $P=[x_0,y_0,z_0,w_0]$ of any pair of curves $C_i$ and $C_j$ with $i\neq j$, and we will give an explicit formula for the sum
\[
\sum_{i\neq j}I_P(C_i,C_j),
\]
where $I_P(C_i,C_j)$ denotes the intersection multiplicity of $C_i$ and $C_j$ at the point $P$. 

 From $P\in X_d$, we have $x_0^d+y_0^d+z_0^d+w_0^d=0$, hence not all $x_0,y_0,z_0$ are zero. By the symmetry of the coordinates $x,y,z$ in the equation $X_d: x^d+y^d+z^d+w^d=0$, we may without loss of generality assume that $x_0\neq0$. To decide where the intersection point $P=[x_0,y_0,z_0,w_0]$ may situate on $X_d$, we propose the following lemma which will be proved at the end of this section.

\begin{lem}\label{w0}
Following the notations above, for any point $Q$ lying on at least two irreducible components of $\ms{C}$, we have
\[
w(Q)=0,
\]
where $w$ is the fourth homogeneous coordinate for $\bb{P}^3$. Moreover, $Q\in C_i$ for all $i=1,\ldots, k$, i.e., $Q$ lies on every irreducible component of $\ms{C}$.

In particular, for the point $P=[x_0,y_0,z_0,w_0]$ above, we have $w_0=0$ and $P\in C_i$ for all $i\leq k$.
\end{lem}

Let $p=\rho(P)=[x_0,y_0,z_0]$. Then $p\in E\cap D$.

To calculate the intersection multiplicities at $P$, we introduce affine coordinates
\[
u=\frac{y}{x}-\frac{y_0}{x_0},\qquad v=\frac{z}{x}-\frac{z_0}{x_0},\qquad \omega=\frac{w}{x}.
\]
Then the curve $\ms{C}: g=0$ on $X_d$ is defined in the affine piece $\bb{C}^3=\{x\neq0\}\subset\bb{P}^3$ by
\[
h(u,v)=0,\quad l(u,v)+\omega^d=0,
\]
where
\[
h(u,v)=g\left(1,u+\frac{y_0}{x_0},v+\frac{z_0}{x_0}\right)
\]
and
\[
l(u,v)=1+\left(u+\frac{y_0}{x_0}\right)^d+\left(v+\frac{z_0}{x_0}\right)^d.
\]

Moreover, since $x_0^d+y_0^d+z_0^d=0$ and $x_0\neq0$, at least one of $y_0,z_0$ is nonzero. We may without loss of generality assume that $z_0\neq0$. Then 
\[
\frac{\p\left(l(u,v)+\omega^d\right)}{\p v}\Bigg|_{(0,0,0)}=d\left(\frac{z_0}{x_0}\right)^{d-1}\neq0.
\]
 Therefore, we can locally express $v$ as an analytic function of $(u,\omega)$ by applying the implicit function theorem to the equation $l(u,v)
+\omega^d=0$. Then the local equation for $\ms{C}$ can be expressed as $h\left(u,v(u,\omega)\right)=0$. By analyzing this local equation, we can investigate the local branches and compute the local intersection multiplicities. This is the usual way to analyze the curve $\ms{C}$.

Our method is, however, first to solve $v$ from the equation $h(u,v)=0$ using Puiseux series, say $v=s(u)$, and then study the properties of the curve $\ms{C}$ and calculate the intersection multiplicities using the equation $l(u,s(u))+\omega^d=0$. To this end, we need that $h$ is $v$-general, i.e., $h(0,v)$ is not identically zero as a function in $v$.

At this point, we specially consider the exceptional case where $h$ is not $v$-general, or equivalently, $h(0,v)\equiv0$; the $v$-general case will be investigated in Section \ref{sec:intnum}.

Since $h$ is transformed by an affine coordinate change from an irreducible polynomial $g$, it is an irreducible polynomial in $u,v$. If $h$ is not $v$-general, then $u$ is a factor of $h$; therefore, by irreducibility, $h$ is a constant multiple of $u$. Transforming $h$ back to $g$, we see that $g$ is a linear function $y-ax$ for some $a\in\bb{C}$ (up to multiplicative factor). As a result, the curve $\ms{C}$ is defined as
\[
\ms{C}: y=ax,\ (1+a^d)x^d+z^d+w^d=0.
\]

If $a^d+1\neq0$, then $\ms{C}$ is itself irreducible and hence it must be equal to $C$. We have $C^2=d$ in this case.

If $a^d+1=0$, then $\ms{C}$ splits into $d$ lines. Let 
\[
L_i: y=ax,\ z=e^{(1+2i)\pi\sqrt{-1}/d}w,
\]
for $i=1,\ldots, d$, then
\[
(g)=L_1+\cdots+L_d.
\]
It is straightforward that $L_i\cdot L_j=1$ for $i\neq j$. Hence, by the above discussions, we have
\[
d=dC^2+d(d-1),
\]
and so $C^2=2-d$.

\subsection{Proof of Lemma \ref{w0} }Recall that we have the degree $d$ branched Galois cover $\rho:X_d\to\bb{P}^2$ defined by $[x,y,z,w]\mapsto[x,y,z]$. The branch locus of $\rho$ is the Eermat curve $D:x^d+y^d+z^d=0$, and the ramification locus of $\rho$ is the inverse image $\rho^{-1}(D)$ which is defined by the equation $w=0$ on $X_d$.

Assume now that $Q$ is a point on at least two components of $\ms{C}$. Denote $q=\rho(Q)$. Then $\#\left(\rho^{-1}(q)\right)<d$, where $\# A$ means the number of points in a set $A$. Hence $q$ lies on the branch locus $D$ and then $Q\in\rho^{-1}(D)$, so $w(Q)=0$ as desired.
Furthermore, we can easily see that $\rho^{-1}(q)=\{Q\}$.  For all $i\leq k$, the restriction morphism $\rho|_{C_i}: C_i\to E$ is a finite morphism and thus $\rho^{-1}(q)\cap C_i\neq\emptyset$. But $\rho^{-1}(q)$ consists of only one point $Q$, it follows that $Q\in C_i$. We are done.

\section{Self-intersection number for curves on Fermat surfaces}

Following the notations above, we will give an explicit formula for the self-intersection number $C^2$. 

\subsection{Position of intersection points} 
From Lemma \ref{w0}, we see that any intersection point $Q$ of two irreducible components of $\ms{C}$ satisfies $w(Q)=0$. It follows that $q=\rho(Q)$ is an intersection point of $D$ and $E$. The converse is also true.
\begin{lem}\label{pos}
If $q=[x,y,z]$ lies on both $D$ and $E$, then $Q=[x,y,z,0]$ lies on all of $C_i$ for $i\leq k$.

Hence, $Q\in C_i\cap C_j$ for $i\neq j$ if and only if $q\in D\cap E$.
\end{lem}
\begin{proof}
By definition, $\rho^{-1}(q)=\{Q\}$ and $Q\in\ms{C}$. Since each $C_i$ is mapped  onto $E$ under the branched covering map $\rho:X_d\to\bb{P}^2$, the intersection $\rho^{-1}(q)\cap C_i$ is nonempty. Hence $Q\in C_i$ as desired. 

The last assertion follows immediately by Lemma \ref{w0}.
\end{proof}

\subsection{Intersection multiplicities}
From formula \eqref{eq:C2} and the discussions in Section \ref{sec:str}, to obtain a lower bound for $C^2$, we seek to estimate the sum
\[
\sum_{i\neq j}I_P(C_i,C_j),
\]
where $P=[x_0,y_0,z_0,0]$ with $p=[x_0,y_0,z_0]\in D\cap E$. Our key technical results are summarized as follows.
\begin{prop}\label{intnum}
With the notations above, we have
\begin{enumerate}[{\rm (A)}]
\item $m=1$; in other words, $\ms{C}$ is a reduced scheme.
\item The following equality holds:
\[
\mu_P(\ms{C})=(d-1)i_p(D,E)+d\mu_p(E)-d+1,
\]
where $i_p(D,E)$ denotes the intersection multiplicity of the plane curves $D$ and $E$ at the point $p$.
\item We have
\[
\sum_{i\neq j}I_P(C_i,C_j)=(d-1) i_p(D,E)+d\mu_p(E)-k\mu_P(C)-(d-k),
\]
and thus
\[
\sum_{i\neq j}C_i\cdot C_j=d(d-1)e+d\sum_{p\in D\cap E}\mu_p(E)-k\sum_{p\in D\cap E}\mu_{\rho^{-1}(p)}(C)-(d-k)\#(D\cap E).
\]
\end{enumerate}
\end{prop}

The proof will be presented in Section \ref{sec:intnum}.

Now we use the above results to study the Bounded Negativity Conjecture for Fermat surfaces. It is clear that we only need to consider negative curves $C$. Then the curve $\ms{C}$ necessarily has $k\geq2$ irreducible components since otherwise $C^2$ would be nonnegative. Moreover, the number $k$ divides $d$ since the Galois action for $\rho$ is a cyclic action of order $d$. To ensure that the polynomial $h$ is $v$-general, we only need to assume $C^2\neq 2-d$; see the discussions in Section \ref{sec:str}. Therefore, from Lemma \ref{curvST} and Proposition \ref{intnum} together with the equality
\[
kC^2=de^2-\sum_{i\neq j}C_i\cdot C_j,
\]
we get the following theorem.
\begin{thm}[Theorem \ref{BNCFermat}]\label{main}
Let $d\geq5$. Suppose that $C$ is a reduced and irreducible curve on $X_d$ with self-intersection number $C^2<2-d<0$. Then the following properties hold.
\begin{enumerate}[(1)]
\item There exists a unique reduced and irreducible curve $E\subset\bb{P}^2$ defined by a degree $e$ polynomial $g(x,y,z)$, such that the cone in $\bb{P}^3$ determined by $g$ contains $C$. 
\item There exists $k\geq2$ and $k|d$ such that
\begin{align*}
C^2=&\frac{d}{k}e^2-\frac{d(d-1)}{k}e-\frac{d}{k}\sum_{p\in D\cap E}\mu_p(E)\\
   &+\sum_{[x,y,z]\in D\cap E}\mu_{[x,y,z,0]}(C)+\left(\frac{d}{k}-1\right)\#(D\cap E).
\end{align*}
In particular, when $d$ is prime, we have $k=d$ and
\[
C^2=e^2-(d-1)e-\sum_{p\in D\cap E}\mu_p(E)+\sum_{[x,y,z]\in D\cap E}\mu_{[x,y,z,0]}(C).
\]
\end{enumerate}
\end{thm}

As a corollary, we get two sufficient conditions for the Bounded Negativity Conjecture to hold on Fermat surfaces.

\begin{cor}[Corollary \ref{cor:sufcond}]\label{sufcond}
With the notations above, the Bounded Negativity Conjecture holds on the Fermat surface $X_d$ if one of the following conditions is satisfied:
\begin{enumerate}[(a)]
\item
\[
\liminf_{e\to\infty}\inf_{E}\left(e^2-(d-1)e-\sum_{p\in D\cap E}\mu_p(E)\right)\geq -\lambda
\]
for some constant $\lambda>0$ which may depend on $d$; or
\item
\[
\limsup_{e\to\infty}\sup_E\frac{\sum\limits_{p\in D\cap E}\mu_p(E)}{e^2}<1;
\]
\end{enumerate}
where, in the above two formulae, the curve $E$ ranges over all irreducible projective plane curves of degree $e$.
\end{cor}
\begin{proof}
If (a) holds, there exists $e_0>d+1$ such that
\[
\inf_{E}\left(e^2-(d-1)e-\sum_{p\in D\cap E}\mu_p(E)\right)\geq-(\lambda+1)
\]
for $e\geq e_0$. On the other hand, the geometric genus $g(E)$ of $E$ is given by
\[
g(E)=\frac{(e-1)(e-2)}{2}-\sum_{p\in E}\frac{1}{2}\left(\mu_p(E)+r_p-1\right)
\]
by \cite[Cor.7.1.3]{Wal0401}, where $r_p$ is the number of local branches of $E$ at $p$. It follows that
\[
\sum_{p\in D\cap E}\mu_p(E)\leq\sum_{p\in E}\mu_p(E)\leq (e-1)(e-2)
\]
Hence, for $e<e_0$, we have
\begin{eqnarray*}
e^2-(d-1)e-\sum_{p\in D\cap E}\mu_p(E)&\geq& e^2-(d-1)e-(e-1)(e-2)\\
                                      &=&(4-d)e-2\\
                                      &\geq&-(d-4)e_0-2.
\end{eqnarray*}
Therefore, we may take
\[
b_{X_d}=\max\{\lambda+1,(d-4)e_0+2\}.
\]

The condition (b) implies (a), as can be seen as follows: if (b) holds, then for $e$ sufficiently large, we have
\[
e^2-(d-1)e-\sum_{p\in D\cap E}\mu_p(E)\geq e^2-(d-1)e-(1-\al)e^2=\al e^2-(d-1)e.
\]
for some $\al>0$. The desired result follows from 
\[
\al e^2-(d-1)e=\al\left(e-\frac{d-1}{2\al}\right)^2-\frac{(d-1)^2}{4\al}\geq-\frac{(d-1)^2}{4\al}.
\]
\end{proof}

\subsection{An invariant} From Corollary \ref{sufcond} (b), we are led to consider the following invariant
\[
\limsup_{e\to\infty}\sup_E\frac{\sum\limits_{p\in D\cap E}\mu_p(E)}{e^2}.
\]
More generally, given any reduced homogeneous polynomial $f$ of degree $d$, we define
\[
I_f=\limsup_{e\to\infty}\sup_{g\ \mathrm{irreducible}}\frac{\sum\limits_{p\in V(f)\cap V(g)}\mu_p\left(V(g)\right)}{e^2},
\]
where $V(g)$ denotes the projective plane curve defined by $g$. We may ask the following questions which are possibly of independent interest in singularity theory of projective plane curves.

\begin{prob}
\begin{enumerate}
\item  If $f=x^d+y^d+z^d$, is it true that $I_f<1$?
\item More generally, if $f$ defines a smooth projective curve in $\bb{P}^2$, does it hold that $I_f<1$?
\item Does the invariant $I_f$ depend on the equation of $f$, or at least the degree $d$ of $f$?
\end{enumerate}
\end{prob}

\section{Puiseux series, resultants, and intersection multiplicities}\label{sec:intnum}
In this section, we give the proof of Proposition \ref{intnum}, by using Puiseux series and resultants. Since these techniques have become standard tools in algebraic curve theory, with well-established references available. For conciseness, we just introduce simplified notations adapted to our context and refer to standard books for comprehensive definitions and properties.

\subsection{Puiseux series} For standard definitions, notations and properties of Puiseux series, we refer to \cite[Chap.1]{Cas0001}. For convenience, for a non-zero Puiseux series
\[
s(u)=\sum_{i=1}^\infty a_iu^{\frac{i}{n}},
\]
with polydromy order $\nu(s)=n$, we set
\[
s_{,\al}(u)=\sum_{i=1}^\infty\exp\left(\frac{2\pi\sqrt{-1}\al i}{n}\right)a_iu^{\frac{i}{n}}
\]
for $\al=1,\ldots, n$; this notation will be used to simplify the expression of decomposing an irreducible polynomial into a product of Puiseux series. Moreover, for $s,s'\in\bb{C}\llangle u\rrangle$, denote $s\simeq s'$ if
\[
s=(a+s'')s'
\]
for some $a\in\bb{C}^*$ and $s''$ a Puiseux series; we will only use the property that $s\simeq s'$ implies $o_u(s)=o_u(s')$ (here, $o_u(s)$ denotes the order of $s$ in $u$ \cite[Sec.1.2]{Cas0001}).

\subsection{Resultants, discriminants, and intersection multiplicities} The references abound. For the definition and basic properties, the interested reader may consult \cite[Sec.1.4]{BenRis9001}, \cite[Chap.12]{GKZ94},\cite[Chap.2]{Cas0001}, \cite[Chap.I, Sec.3]{GreLosShu0701}, and \cite[Chap.1]{Wal0401}. The notations used in different references varying a little, we will fix notations that are suited for our purpose.

Suppose $F,G$ are two reduced curve germs in $(\bb{C}^2,0)$, defined by
\[
F:\ f(u,v)=0,
\]
and
\[
G:\ g(u,v)=0.
\]
By the Weierstrass preparation theorem (see \cite[Thm.1.8.7]{Cas0001}), we may assume that $f,g$ are Weierstrass polynomials:
\[
f=v^d+a_1(u)v^{d-1}+\cdots+a_d(u),
\]
and
\[
g=v^e+b_1(u)v^{e-1}+\cdots+b_e(u),
\]
where $a_i, b_j\in\bb{C}\{u\}$ satisfy $a_i(0)=b_j(0)=0$ for $1\leq i\leq d,1\leq j\leq e$. For convenience, we will regard $a_i,b_j$ as elements of $\bb{C}\llangle u\rrangle$ by the inclusion $\bb{C}\{u\}\subset\bb{C}\llangle u\rrangle$, and thus regard $f,g$ as polynomials in $\bb{C}\llangle u\rrangle[v]$.

The \textit{resultant} $\Res(f,g)$ is
\[
\Res(f,g)=\prod_{\substack{1\leq i\leq d\\ 1\leq j\leq e}}(\lambda_i-\mu_j),
\]
where $\lambda_1,\cdots,\lambda_d$ and $\mu_1,\cdots,\mu_e$ are respectively the roots of $f,g$ in $\bb{C}\llangle u\rrangle$. In addition, the \textit{discriminant} $\Disc(f)$ is 
\[
\Disc(f)=\prod_{i<j}(\lambda_i-\lambda_j)^2.
\]

The \textit{intersection multiplicity} of $F$ and $G$ at $0$, denoted by $i_0(F,G)$, which is equal to the intersection multiplicity of $f$ and $g$ at $0$, denoted by $i_0(f,g)$, is 
\[
i_0(F,G)=i_0(f,g)=\dim_{\bb{C}}\frac{\bb{C}\{u,v\}}{(f,g)}=o_u\left(\Res(f,g)\right),
\]
where the last equality is the Halphen formula \cite[Prop.2.6.3]{Cas0001}. In addition, the Milnor number of $F$ at $0$, denoted by $\mu_0(F)$ or $\mu_0(f)$, is
\[
\mu_0(F)=\mu_0(f)=i_0\left(\frac{\p f}{\p u},\frac{\p f}{\p v}\right)=\dim_{\bb{C}}\frac{\bb{C}\{u,v\}}{\left(\frac{\p f}{\p u},\frac{\p f}{\p v}\right)}.
\]

We establish the relationship between the Minor number and the discriminant.
\begin{lem}\label{dis}
With the notations above, we have
\[
o_u\left(\Disc(f)\right)=\mu_0(f)+d-1.
\]
\end{lem}
\begin{proof}
By \cite[Lem.6.5.7]{Wal0401}, we have that the number 
\[
i_0\left(pf+qu\frac{\p f}{\p u},\frac{\p f}{\p v}\right)
\]
is independent of $(p,q)$ provided $p,q\geq0$ and $(p,q)\neq(0,0)$. It follows that
\[
o_u\left(\Disc(f)\right)=i_0\left(f,\frac{\p f}{\p v}\right)=i_0\left(u\frac{\p f}{\p u},\frac{\p f}{\p v}\right)=i_0\left(u,\frac{\p f}{\p v}\right)+\mu_0(f).
\]
Since $f=v^d+\sum a_i(u)v^{d-i}$, we get
\[
i_0\left(u,\frac{\p f}{\p v}\right)=i_0(u, v^{d-1})=d-1.
\]
Hence, we get the desired equality $o_u\left(\Disc(f)\right)=\mu_0(f)+d-1$.
\end{proof}

\subsection{Proof of Proposition \ref{intnum}}
Following the notations in Section \ref{sec:str}, assume $h$ is $v$-general of order $b$, i.e.,
\[
h(0,v)=c_b v^b+c_{b+1}v^{b+1}+\cdots,
\]
for some constant $c_b\neq0$. 

We first decompose $h$ in $\bb{C}\{u,v\}$ into irreducible factors
\[
h= h_0 h_1\cdots h_r;
\]
where $h_1,\ldots, h_r$ are irreducible Weierstrass polynomial in $\bb{C}\{u\}[v]$ of degrees $b_1,\ldots, b_r$ respectively, and $h_0$ is an invertible element in $\bb{C}\{u,v\}$, and $b=b_1+\cdots+b_r$.

For $1\leq i\leq r$, let $s_i$ be a Puiseux series which is a root of $h_i$. By \cite[Cor.1.5.8, Thm.1.8.3]{Cas0001}, we have $\nu(s_i)=b_i$ and
\[
h_i=\prod_{\al=1}^{b_i}\left(v-s_{i,\al}(u)\right),\ 1\leq i\leq r.
\]

Set
\[
\phi_i(u,\omega)=\prod_{\al=1}^{b_i}\left(\omega^d+l\left(u,s_{i,\al}(u)\right)\right),\ 1\leq i\leq r,
\]
and
\[
\Phi(u,\omega)=\phi_1\cdots\phi_r.
\]
Then $\Phi$ is a Weierstrass polynomial in $\bb{C}\{u\}[\omega]$ of degree $bd$ and moreover, the scheme $\ms{C}$ is locally defined by the equation $\Phi(u,\omega)=0$ in the local coordinate system $(u,\omega)$ on $X_d$.

(A) To show that $\ms{C}$ is reduced, it is sufficient to show 
\begin{enumerate}
\item[($a_1$)] For $1\leq i\leq r$, $\phi_i$ is reduced; and
\item[($a_2$)] For $1\leq i<j\leq r$, $\phi_i$ and $\phi_j$ are coprime as elements in $\bb{C}\llangle u\rrangle[\omega]$.
\end{enumerate}
For ($a_1$),  by definition of the intersection multiplicity, we have
\[
o_u\left(l\left(u,s_{i,\al}(u)\right)\right)=\frac{i_0(h_i,l)}{b_i}\leq\frac{i_p(D,E)}{b_i}<\infty.
\]
In particular, $l\left(u,s_{i,\al}(u)\right)\neq0$; and thus $\omega^d+l\left(u,s_{i,\al}(u)\right)$ is reduced. Moreover, recall that 
\[
l(u,v)=1+\left(u+\frac{y_0}{x_0}\right)^d+\left(v+\frac{z_0}{x_0}\right)^d,
\] 
hence, for $\al\neq\be$,
\begin{eqnarray*}
 & &l\left(u,s_{i,\al}(u)\right)-l\left(u,s_{i,\be}(u)\right)\\
 &=&\left(s_{i,\al}(u)-s_{i,\be}(u)\right)\left(d\left(\frac{z_0}{x_0}\right)^{d-1}+r(u)\right)\\
 &\simeq&s_{i,\al}(u)-s_{i,\be}(u)
\end{eqnarray*}
where $\lim\limits_{u\to0}r(u)=0$ and $\simeq$ means equality up to multiplication by an invertible element. Since $\nu(s_i)=b_i$ and $\al\neq\be$, we have
\[
s_{i,\al}(u)-s_{i,\be}(u)\neq0,
\]
and thus
\[
\gcd\left(\omega^d+l\left(u,s_{i,\al}(u)\right),\omega^d+l\left(u,s_{i,\be}(u)\right)\right)=1.
\]
Consequently, $\phi_i$ is reduced.

For ($a_2$), it suffices to show that for any $i\neq j$ and for all $\al\leq b_i,\be\leq b_j$,
\[
l\left(u,s_{i,\al}(u)\right)\neq l\left(u,s_{j,\be}(u)\right).
\]
Similar to the discussions above, we have
\[
l\left(u,s_{i,\al}(u)\right)-l\left(u,s_{j,\be}(u)\right)\simeq s_{i,\al}(u)-s_{j,\be}(u).
\]
Since $h_i$ and $h_j$ are coprime, they do not have common roots in $\bb{C}\llangle u\rrangle$; therefore $s_{i,\al}(u)-s_{j,\be}(u)\neq0$.

\begin{rk}
Proposition \ref{intnum} (A) also follows from the fact that for all $q\in E\setminus(E\cap D)$, the inverse image $\rho^{-1}(q)\subset\ms{C}$ consists of exactly $d$ points.
\end{rk}

(B) To compute $\mu_P(\ms{C})$, we use Lemma \ref{dis} and get
\[
\mu_P(\ms{C})=o_u\left(\Disc(\Phi)\right)-bd+1.
\]
Hence, we need to calculate $o_u\left(\Disc(\Phi)\right)$. 

From $\Phi=\phi_1\cdots\phi_r$, we have
\[
\Disc(\Phi)=\left(\prod_{i=1}^r\Disc(\phi_i)\right)\left(\prod_{1\leq i<j\leq r}\Res(\phi_i,\phi_j)^2\right),
\]
and thus
\[
o_u\left(\Disc(\Phi)\right)=\sum_{i=1}^ro_u\left(\Disc(\phi_i)\right)+\sum_{1\leq i<j\leq r}o_u\left(\Res(\phi_i,\phi_j)^2\right).
\]
Similarly, for $1\leq i\leq r$,
\begin{eqnarray*}
 & &o_u\left(\Disc(\phi_i)\right)\\
 &=&\sum_{\al=1}^{b_i}o_u\left(\Disc\left(\omega^d+l\left(u,s_{i,\al}(u)\right)\right)\right)\\
 & &+\sum_{1\leq\al<\be\leq b_i}o_u\left(\Res\left(\omega^d+l\left(u,s_{i,\al}(u)\right),\omega^d+l\left(u,s_{i,\be}(u)\right)\right)^2\right).
\end{eqnarray*}

We will prove the following results.
\begin{enumerate}
\item[($b_1$)] For $1\leq i\leq r$, 
\[
\sum_{\al=1}^{b_i}o_u\left(\Disc\left(\omega^d+l\left(u,s_{i,\al}(u)\right)\right)\right)=(d-1)i_0(h_i,l),
\]
\item[($b_2$)] For $1\leq i\leq r$,
\begin{eqnarray*}
 & &\sum_{1\leq\al<\be\leq b_i}o_u\left(\Res\left(\omega^d+l\left(u,s_{i,\al}(u)\right),\omega^d+l\left(u,s_{i,\be}(u)\right)\right)^2\right)\\
 &=&do_u\left(\Disc(h_i)\right)\\
 &=&d\left(\mu_0(h_i)+b_i-1\right).
\end{eqnarray*}
\item[($b_3$)] For $1\leq i<j\leq r$,
\[
o_u\left(\Res(\phi_i,\phi_j)^2\right)=do_u\left(\Res(h_i,h_j)^2\right).
\]
\item[($b_4$)] We have
\[
o_u\left(\Disc(\Phi)\right)=(d-1)i_p(D,E)+d\left(\mu_p(E)+b-1\right).
\]
\end{enumerate}

For ($b_1$), we have
\begin{eqnarray*}
\Disc\left(\omega^d+l\left(u,s_{i,\al}(u)\right)\right)&\simeq&\Res\left(\omega^d+l\left(u,s_{i,\al}(u)\right), \omega^{d-1}\right)\\
                                                                         &\simeq&\left(l\left(u,s_{i,\al}(u)\right)\right)^{d-1};
\end{eqnarray*}
hence,
\begin{eqnarray*}
o_u\left(\Disc\left(\omega^d+l\left(u,s_{i,\al}(u)\right)\right)\right)&=&(d-1)o_u\left(l\left(u,s_{i,\al}(u)\right)\right)\\
                                                                                         &=&\frac{(d-1)i_0(h_i,l)}{b_i}.
\end{eqnarray*}
Summing over $\al=1,\ldots, b_i$ gives the desired equality.

For ($b_2$), we have
\begin{eqnarray*}
& &\Res\left(\omega^d+l\left(u,s_{i,\al}(u)\right),\omega^d+l\left(u,s_{i,\be}(u)\right)\right)\\
&=&\Res\left(\omega^d+l\left(u,s_{i,\al}(u)\right),l\left(u,s_{i,\be}(u)\right)-l\left(u,s_{i,\al}(u)\right)\right)\\
&\simeq&\left(l\left(u,s_{i,\be}(u)\right)-l\left(u,s_{i,\al}(u)\right)\right)^d.
\end{eqnarray*}
In the proof of (A), we have shown that
\[
l\left(u,s_{i,\be}(u)\right)-l\left(u,s_{i,\al}(u)\right)\simeq s_{i,\be}(u)-s_{i,\al}(u).
\]
Hence,
\begin{eqnarray*}
& &\Res\left(\omega^d+l\left(u,s_{i,\al}(u)\right),\omega^d+l\left(u,s_{i,\be}(u)\right)\right)\\
&\simeq&\left(s_{i,\be}(u)-s_{i,\al}(u)\right)^d.
\end{eqnarray*}
It follows that
\begin{eqnarray*}
& &\sum_{1\leq\al<\be\leq b_i}o_u\left(\Res\left(\omega^d+l\left(u,s_{i,\al}(u)\right),\omega^d+l\left(u,s_{i,\be}(u)\right)\right)^2\right)\\
&=&d\sum_{1\leq\al<\be\leq b_i}o_u\left(s_{i,\be}(u)-s_{i,\al}(u)\right)^2\\
&=&d o_u\left(\Disc(h_i)\right)\\
&=&d o_u\left(\mu_0(h_i)+b_i-1\right),
\end{eqnarray*}
where the last equality follows from Lemma \ref{dis}.

For ($b_3)$, we have
\begin{eqnarray*}
 & &\Res(\phi_i,\phi_j)^2\\
 &=&\prod_{\al=1}^{b_i}\prod_{\be=1}^{b_j}\Res\left(\omega^d+l\left(u,s_{i,\al}(u)\right),\omega^d+l\left(u,s_{j,\be}(u)\right)\right)^2\\
 &\simeq&\prod_{\al=1}^{b_i}\prod_{\be=1}^{b_j}\Res\left(\omega^d+l\left(u,s_{i,\al}(u)\right),l\left(u,s_{j,\be}(u)\right)-l\left(u,s_{i,\al}(u)\right)\right)^2\\
 &\simeq&\prod_{\al=1}^{b_i}\prod_{\be=1}^{b_j}\left(l\left(u,s_{j,\be}(u)\right)-l\left(u,s_{i,\al}(u)\right)\right)^{2d}\\
 &\simeq&\prod_{\al=1}^{b_i}\prod_{\be=1}^{b_j}\left(s_{j,\be}(u)-s_{i,\al}(u)\right)^{2d}\\
 &\simeq&\Res(h_i,h_j)^{2d}.
\end{eqnarray*}
It follows that
\[
o_u\left(\Res(\phi_i,\phi_j)^2\right)=do_u\left(\Res(h_i,h_j)^2\right),
\]
as desired.

For ($b_4$), we combine ($b_1$)--($b_3$) and get
\begin{eqnarray*}
o_u\left(\Disc(\Phi)\right)&=&\sum_{i=1}^r\left((d-1)i_0(h_i,l)+do_u\left(\Disc(h_i)\right)\right)+\sum_{1\leq i<j\leq r}\left(do_u\left(\Res(h_i,h_j)^2\right)\right)\\
                           &=&(d-1)\sum_{i=1}^ri_0(h_i,l)+do_u\left(\Disc(h)\right)\\
                           &=&(d-1)i_0(h,l)+do_u\left(\Disc(h)\right).
\end{eqnarray*}
Note that $i_0(h,l)=i_p(E,D)$; and by Lemma \ref{dis}, we have
\[
 o_u\left(\Disc(h)\right)=\mu_p(E)+b-1.
\]
The equality in ($b_4$) follows immediately.

Now we are ready to finish the proof of (B). From ($b_4$) and the equality
\[
o_u\left(\Disc(\Phi)\right)=\mu_P(\ms{C})+bd-1,
\]
we get
\begin{eqnarray*}
\mu_P(\ms{C})&=&(d-1)i_p(D,E)+d\left(\mu_p(E)+b-1\right)-bd+1\\
             &=&(d-1)i_p(D,E)+d\mu_p(E)-d+1.
\end{eqnarray*}

(C) Since $C_1,\ldots, C_k$ are the irreducible components of $\ms{C}$, we have from \cite[Thm.6.5.1]{Wal0401} that
\[
\mu_P(\ms{C})=\sum_{i\neq j}I_P(C_i,C_j)+\sum_{i=1}^k\mu_P(C_i)-k+1.
\]
Note that $P$ is a fixed point under the action of the Galois group for the branched cover $\rho$, and the Galois group acts transitively on $\{C_1,\ldots,C_k\}$, hence
\[
\mu_P(C_i)=\mu_P(C),\ \mathrm{for\ all\ }i=1,\ldots, k.
\]
It follows that
\[
\sum_{i\neq j}i_P(C_i,C_j)=\mu_P(\ms{C})-k\mu_P(C)+k-1.
\]
Using (B), we get
\[
\sum_{i\neq j}i_P(C_i,C_j)=(d-1)i_p(D,E)+d\mu_p(E)-k\mu_P(C)-(d-k).
\]
Summing over all possible intersection points $P\in C_i\cap C_j$ with $i\neq j$ and using Lemma \ref{pos}, we obtain
\[
\sum_{i\neq j}C_i\cdot C_j=(d-1) D\cdot E+d\sum_{p\in D\cap E}\mu_p(E)-k\sum_{p\in D\cap E}\mu_{\rho^{-1}(p)}(C)-(d-k)\#(D\cap E).
\]
Note that $D\cdot E=de$, the equality in (C) follows immediately. 

The proof of Proposition \ref{intnum} is complete now.

\bigskip
{\bf Data Availability: }This manuscript has no associated data. All results are derived through rigorous mathematical proofs.

\bigskip
{\bf Funding Declaration: }This work was previously supported by the Chinese Academy of Sciences (CAS) under Grant No. YSBR-001.


\end{document}